\documentclass[a4paper,12pt,fleqn,leqno]{article}

\usepackage{amssymb,amsmath,amsthm,amsfonts}
\usepackage{tikz}
\usepackage[colorlinks=true,allcolors=blue]{hyperref}
\usepackage[verification]{struktex}

\newcommand {\<}{{\fontencoding{T1}\selectfont<}}

\newcommand {\bx} {{\pmb{x}}}
\newcommand {\C} {{\mathbb C}}
\newcommand {\card} {\#}
\newcommand {\cat}{\;\!{}_{{}^{{}^{{}^\circ}}}}

\newcommand {\cP} {{\cal P}}
\newcommand {\cS} {{\cal S}}

\newcommand {\comment}[1] {}
\newcommand {\content}[1] {{|#1|}}
\newcommand {\dder}{{\textstyle\SquareDel}}
\newcommand {\dint}{{\textstyle\SquareInt}}
\newcommand {\df}[1] {{\bfseries #1}}
\newcommand {\qt}[1] {``#1''}

\newcommand {\enum}[2] {\{#1,\dots,#2\}}

\newcommand {\FZ} {{F\Z}}
\makeatletter\newcommand{\hathat}[1]{% 
\begingroup%
  \let\macc@kerna\z@%
  \let\macc@kernb\z@%
  \let\macc@nucleus\@empty%
  {\hat{\raisebox{.3ex}{\vphantom{\ensuremath{#1}}}\smash{\hat{#1}}}}%
\endgroup%
}\makeatother
\newcommand {\gol}{\gamma}
\newcommand {\Gol}{\mathit{\Gamma}}
\newcommand {\id}{\mathrm{id}}
\newcommand {\into} {\hookrightarrow}
\newcommand {\la} {\lambda}
\newcommand {\len}[1] {{\#{#1}}}

\newcommand {\lev}{\ell}
\newcommand {\Lev}{L}
\newcommand {\longrev}[1] {{\overleftarrow{#1}}}
\newcommand {\mi} {\text{-}}
\newcommand {\minus} {\text{\!\scalebox{0.7}{$-$}}}
\newcommand {\N} {\mathbb N}
\newcommand {\OEIS}[1] {\href{https://oeis.org/A#1}{A#1}}
\newcommand {\oeitem}[1] {\item[\OEIS{#1}]}

\newcommand {\ot} {\leftarrow}
\newcommand {\ov}[1] {{\overline{#1}}} 

\newcommand {\ph} {\varphi}

\newcommand {\Q} {{\mathbb Q}}

\newcommand {\rest}[1] {|_{#1}}
\newcommand{\cev}[1]{{\reflectbox{\ensuremath{\vec{\reflectbox{\ensuremath{#1}}}}}}}
\newcommand {\rev}[1] {{\cev{#1}}}

\newcommand {\set}[2] {{\{ #1 \sothat #2 \}}}

\newcommand {\sm}{\setminus}
\newcommand {\sothat} {\,:\,}

\newcommand {\threecol}[3]{{\left(\!\begin{smallmatrix}#1\\#2\\#3\end{smallmatrix}\!\right)}}
\newcommand {\ti}[1] {{\tilde{#1}}}

\newcommand {\tn}[1] {\tiny$\scriptstyle#1$}
\newcommand {\thru}[3] {#1_{#2},\dots,#1_{#3}}

\newcommand {\Ver}[1]{{V_\minus^{#1}}}
\newcommand {\Z} {\mathbb Z}

\DeclareMathOperator{\Grp}{{\bf Grp}}

\DeclareMathOperator{\Set}{{\bf Set}}

\DeclareMathOperator*{\SquareInt}{\tikz[baseline=0.25em]{\draw[line width=0.06em]
  (-0.2em,0.1em) |- (0,0em) -- (0,1em) -| (0.2em,.9em);}\:\!\!}
\DeclareMathOperator*{\SquareDel}{\tikz{\draw[line width=0.06em]
  (0.2em,0.3em) -| (-0.2em,0em) -| (0.2em,0.6em) -| (-0.1em,.5em);}}

\newcounter{equivalence}
\renewcommand{\theequivalence}{(\roman{equivalence})}

\newcounter{substate}
\renewcommand{\thesubstate}{(\alph{substate})}
\newenvironment{substate}
{
  \begin{list}{\bf\thesubstate}
    {\usecounter{substate}
     \itemindent0em
     \settowidth\labelwidth{\bf(g)} \labelsep0.5em  
     \leftmargin\labelwidth \addtolength\leftmargin\labelsep
     \topsep0.5ex
     \itemsep0ex}
}
{\end{list}}

\newenvironment {subproof}
{
  \begin{list}{\bf\thesubstate}
    {\usecounter{substate} 
     \leftmargin0em 
     \settowidth\labelwidth{\bf(a)} \labelsep0.5em  
     \itemindent\labelwidth \addtolength\itemindent\labelsep
     \topsep0.5ex
     \itemsep0ex}
}
{\end{list}}

\newenvironment{OEISentries}
{\section*{OEIS Entries}\label{sec-oeis}
  \begin{list}{\bf\thesubstate}
    {\leftmargin0em 
     \settowidth\labelwidth{A012345} \labelsep0.5em  
     \leftmargin\labelwidth \addtolength\leftmargin\labelsep
     \topsep0.5ex
     \itemsep0ex}
}
{\end{list}}

\swapnumbers
\newtheorem{thm}{Theorem}[section]
\newtheorem{rmk}[thm]{Remark}
\newtheorem{lem}[thm]{Lemma}
\newtheorem{prp}[thm]{Proposition}
\newtheorem{cor}[thm]{Corollary}
\newtheorem{cnj}[thm]{Conjecture}

\theoremstyle{definition}

\title{Golombic and Levine sequences}
\author{Johan Claes, Roland Miyamoto}

\begin{document}

\maketitle

\begin{abstract}
\noindent
We investigate and generalise Levine sequences
like \OEIS{011784}, \OEIS{061892} and \OEIS{061894}
and develop an algebraic theory for them.
We thereby also cover other fast growing sequences like~\OEIS{014644},
which we call \qt{golombic} due to their strong ties with
Golomb's sequence~\OEIS{001462}.\footnote{%
  Every
  \href{https://oeis.org/wiki/Style_Sheet\#A-number}{A-number}
  links to its corresponding entry in
  \href{https://oeis.org}
  {The On-Line Encyclopedia of Integer Sequences}~\cite{OEIS}
  (OEIS).
  For convenient off-line reading,
  we list all used A-numbers
  before the References,
  each followed by its official
  \href{https://oeis.org/wiki/Style_Sheet\#Name}{OEIS Name}.%
}\\[1ex]
\end{abstract}

\addtocounter{section}{-1}

\section{Introduction}

Three years after Andrew Wiles
had fully proved Fermat's conjecture,
Richard Guy gave his answer to the question \qt{What's left?}
in a short essay~\cite{Gu1} of the same title.
He starts by presenting three \qt{simple arithmetic}
---albeit hard--- problems.
The third of these problems concerns the sequence
\begin{equation}\label{eqn-lev2}
    \lev(2) = \OEIS{011784}
    = (1,2,2,4,7,14,42,213,2837,175450,139759600,6837625106787,\dotsc),
\end{equation}
which Lionel Levine suggested to Neil Sloane in 1997.
It exhibits doubly-exponential growth
and was entered in the OEIS~\cite{OEIS} in 1999.
According to Guy~\cite{Gu1}, \cite[p.~347]{Gu2},
Sloane was wondering
whether its 20th term would ever be calculated.
Sloane listed Levine's sequence~\eqref{eqn-lev2} in~\cite{Sl}
as one of his \qt{favorite integer sequences}
and reiterated his interest in it
in an enthusiastic Numberphile video~\cite{SlHa}
a few years ago.

In May 2025, a single PARI~\cite{PARI} thread
programmed by the authors
calculated the said 20th term,
a 358-digit number, within two hours
when running on an old laptop computer.
The present paper develops the algebraic theory
that underlies this breakthrough.
One main ingredient are certain
integer-valued polynomials $T_n\in\Q[\thru x1n]$
that Vardi~\cite{Va} introduced and used to calculate
far out terms of Golomb's sequence~\eqref{eqn-golomb}.
As explained in Theorem~\ref{thm-gol},
Vardi's poly\-nomials can also be used to efficiently calculate
sequences like~\eqref{eqn-gol2},
which we call \qt{golombic}
owing to their intimate relation with Golomb's sequence.

Every golombic sequence arises from its underlying \qt{golombic triangle},
like~\eqref{eqn-gol2} arises from~\eqref{eqn-gol2triangle}.
Each row of such a triangle is \qt{deployed} (cf.~\cite[p.~7]{Br})
from the previous row by the so-called \qt{golombic operator} $\Gol$.
While these rows traditionally contain only non-negative integers,
Theorem~\ref{thm-Golz} extends the golombic deployment process to 
\qt{words} in the free group $\FZ$ over the integers.

An arbitrary \df{Levine sequence}
(in generalisation of Levine's original sequence~\eqref{eqn-lev2} above)
is produced quite similarly,
only that the rows (words) are reversed between applications of $\Gol$.
Reversing is one of eight natural \mbox{(anti-)involutions} of the free group $\FZ$.
Exploiting these \mbox{(anti-)involutions} and studying
their interplay with the golombic operator (Lemma~\ref{lem-Golinvo})
leads to an iterative formula (Theorem~\ref{thm-lev})
based on Vardi's polynomials $T_n$
for calculating large terms of an arbitrary Levine sequence.
Apart from their algorithmic merit,
Theorems~\ref{thm-gol} and~\ref{thm-lev} also shed some more light
on the \qt{vague similarity} between Golomb's and Levine's sequence
noticed by Guy~\cite[p.~347]{Gu2}.

\section{Integer sequences, golombic and Levine operator}\label{sec-intseq}
An \df{integer sequence} (based at start index $1$)
is a map $a\colon\N\to\Z$;
its \df{length} $\len a$ is the \qt{last} index $n$
for which the term $a(n)$ is non-zero,
more precisely,
\[
    \len a := \sup(\{0\}\cup\set{n\in\N}{a(n)\neq0})
    \; \in \; \N_0\cup\{\infty\},
\]
and we call $a$ \df{finite} or a \df{tuple} if $\len a<\infty$.
For some convenience later on,
we equip the set $\Z^\N$ of all integer sequences
with the group operation $+$ of componentwise addition,
and we will also consider its subgroup of finite sequences
as well as its two submonoids of (finite) non-negative
integer sequences,
\[
    \Z^{(\N)} = \set{a\in\Z^\N}{\len{a}<\infty},
    \quad \N_0^\N
    \quad\text{and}\quad
    \N_0^{(\N)} = \N_0^\N \cap \Z^{(\N)}.
\]
The group homomorphism
$\content{\ }\colon(\Z^{(\N)},+)\to(\Z,+)$
that maps a tuple $a$ to its \df{content}
\[
    \content a := \sum_{n=0}^{\len a}a(n)
\]
obviously restricts to a monoid homomorphism
$\content{\ }\colon(\N_0^{(\N)},+)\to(\N_0,+)$.
A sequence or tuple $a$ is sometimes written
as a parenthesised, comma-separated list of its terms,
that is, $a = (a(1),a(2),\dotsc)$,
where trailing zeros may be omitted.
For example, $(0,0,1,0,0,\dotsc)=(0,0,1,0)=(0,0,1)$
denotes the sequence $a$ with $a(3)=1$ and $a(n)=0$ for $3\neq n\in\N$.
Another useful notation is the \df{left verschiebung} operator
\[
    \Ver{} \colon \Z^\N \to \Z^\N, \quad
    a \mapsto \Ver{}a := (a(2),a(3),\dotsc)
\]
and its powers $\Ver m$ defined by $(\Ver ma)(n)=a(m+n)$
for $a\in\Z^\N$, $m\in\N_0$, $n\in\N$.
Our main attention rests on the operator
$\Gol \colon \N_0^{(\N)} \to \N_0^{(\N)}$
which maps the tuple $a\in \N_0^{(\N)}$ to the tuple
\[
  \Gol a := \left(1^{a(1)}2^{a(2)}3^{a(3)}\!\dotsc\right)
          = \left(1^{a(1)}\!\dotsc(\len a)^{a(\len a)}\right)
          \in \N_0^{(\N)},
\]
meaning that $\Gol a$ consists of:
first $a(1)$ copies of the integer $1$,
then $a(2)$ copies of the integer $2$, and so on,
or, more formally,
\[
    (\Gol a)(n) = \begin{cases}
    k \in \enum1{\len a} & \text{if }
      \sum_{j=1}^{k-1}a(j)<n\leq\sum_{j=1}^ka(j),\\
    0 & \text{if } n>\content a
    \end{cases}
\]
for $n\in\N$.
We observe that $\Gol$ maps the first $n\in\N$ terms of
\df{Golomb's sequence}
\begin{gather}\label{eqn-golomb}
    \OEIS{001462}
      = (1,2,2,3,3,4,4,4,5,5,5,6,6,6,6,7,7,7,7,8,8,8,8,9,9,9,9,9,10,\dotsc)
\end{gather}
to its first $\OEIS{001463}(n)$ terms,
so that $\Gol$ produces arbitrarily many initial terms
of Golomb's sequence by repeated application to, in fact,
any tuple $a\in \N_0^{(\N)}$ other than $()$ or $(1)$.
We therefore call $\Gol$ the \df{golombic operator}.
Let us write $\Gol^n$ for $n$-fold application of $\Gol$.
Starting from the tuple $a=(2)$, we obtain the \qt{triangle}
\begin{align}
\Gol^0(2) &= (2),
          \nonumber\\
\Gol^1(2) &= (1^2)
           = (1,1),
          \nonumber\\
\Gol^2(2) &= (1^1 2^1)
           = (1,2),
          \nonumber\\
\Gol^3(2) &= (1^1 2^2)
           = (1,2,2),
          \label{eqn-gol2triangle}\\
\Gol^4(2) &= (1^1 2^2 3^2)
           = (1,2,2,3,3),
          \nonumber\\
\Gol^5(2) &= (1^1 2^2 3^2 4^3 5^3)
           = (1,2,2,3,3,4,4,4,5,5,5),
          \nonumber\\
\Gol^6(2) &= (1^1 2^2 3^2 4^3 5^3 6^4 7^4 8^4 9^5 10^5 11^5)
           = (1,2,2,3,3,4,\dotsc,
                   10,11,11,11,11,11),\dotsc
          \nonumber
\end{align}
which, except for its first row, is given by \OEIS{014643}.
An apparent and remarkable property of these rows
is that the length of any one row
equals the content of the previous row and also
the la(rge)st number in the following row.
Formally,
\begin{gather}\label{eqn-lenGol}
    \content a = \len\Gol a,
    \quad
    \len a = (\Gol a)(\content a)
    \quad\text{and}\quad
    \content{\Gol a} = \sum_{n=1}^{\len a}n\cdot a(n)
    \quad\text{for every } a \in \N_0^{(\N)}.
\end{gather}
In further pursuit of these ideas,
we introduce the \df{$n$-th golombic number}
of a given tuple $a\in \N_0^{(\N)}$ as
\begin{gather}\label{eqn-golna}
    \gol_na := \len \Gol^{n-1}a
    \quad\text{for }n\in\N
\end{gather}
and call $\gol a := (\gol_na)_{n\in\N}$
the \df{golombic sequence based at} $a$.
The probably most prominent and canonical example is
the golombic sequence
\begin{gather}\label{eqn-gol2}
    \gol(2) = \OEIS{014644}
    = (1,2,2,3,5,11,38,272,6474, 1090483, 4363282578,\dotsc)
\end{gather}
based at $(2)$,
arising from the above \qt{triangle}~\eqref{eqn-gol2triangle}.
Its growth is doubly exponential,
which is probably why only its first 15 terms
have been calculated so far.
The theory and methods developed in this paper
will enable us to compute four more terms,
the last of which has 444 decimal places.

A variation of the golombic operator and sequences
is given by Lionel Levine's construction
which reverses the tuple between applications of $\Gol$.
Denoting the \df{reverse} of a given tuple $a\in\Z^{(\N)}$ by
\[
  \rev a = (a(\len a),\dotsc,a(1)),
\]
we define the \df{Levine operator}
\[
    \Lev\colon\N_0^{(\N)}\to\N_0^{(\N)},\quad
    a \mapsto \longrev{\Gol a}
\]
and write $\Lev^n$ for its $n$-fold application.
Starting from the tuple $a=(2)$ leads to Levine's \qt{triangle}
\begin{align}
\Lev^0(2) &= (2),
          \nonumber\\
\Lev^1(2) &= (2^1)
           = (1,1),
          \nonumber\\
\Lev^2(2) &= (2^1 1^1)
           = (2,1),
          \nonumber\\
\Lev^3(2) &= (2^1 1^2)
           = (2,1,1),
          \label{eqn-lev2triangle}\\
\Lev^4(2) &= (3^1 2^1 1^2)
           = (3,2,1,1),
          \nonumber\\
\Lev^5(2) &= (4^1 3^1 2^2 1^3)
           = (4,3,2,2,1,1,1),
          \nonumber\\
\Lev^6(2) &= (7^1 6^1 5^1 4^2 3^2 2^3 1^4)
           = (7,6,5,4,4,3,3,2,2,2,1,1,1,1),\nonumber\\
\Lev^7(2) &= (14^1 13^1 12^1 11^1 10^2 9^2 8^2 7^3 6^3 5^4 4^4 3^5 2^6 1^7)
           = (14,13,12,11,10,10,9,9,\dotsc),\,\dotsc
          \nonumber
\end{align}
In contrast to the presentations in~\cite{Gu1},
\cite[p.~347]{Gu2},
\OEIS{011784}, \OEIS{012257} and~\cite{Sl},
each row $\Lev^1(2),\Lev^2(2),\dotsc$
is a {\em decreasing} tuple
because that is one standard way of denoting
an (unordered integer) partition
and allows for an alternative description
in terms of the conjugate partition:
We recall (see~\cite[pp.~7f]{An})
that the \df{conjugate} of a \df{partition}
\[
    \la \in \cP := \set{a\in\N_0^{(\N)}}{a(1)\geq\dotsm\geq a(\len a)}
\]
is the partition $\la^*\in\cP$ given by
$\la^*(n) = \#\set{k\in\N}{\la(k)\geq n}$
for $n\in\N$, which moreover satisfies
\[
    \content{\la^*}=\content{\la}
    \quad\text{and}\quad
    \la^{**}=\la.
\]
The following alternative definition for $\Lev$
arises from a combination with the tail-sum operator
$S_\minus\colon\Z^{(\N)}\to\Z^{(\N)}$ given by
\[
    S_\minus a := a+\Ver{}a+\Ver2a+\Ver3a+\dotsm = \sum_{m=0}^{\len a-1}\Ver ma,
    \quad\text{that is},\quad
    (S_\minus a)(n) = \sum_{k=n}^{\len a}a(k)
\]
for $a\in\Z^{(\N)}$, $n\in\N$.

\begin{lem}
  Every tuple $a \in \N_0^{(\N)}$ satisfies
  $S_\minus a\in\cP$ and
  $\Lev a = (S_\minus a)^*$.
\end{lem}
We will not make use of this lemma and leave the proof as an exercise.
For a given tuple $a \in \N_0^{(\N)}$, we obtain
\[
    \content a = \len\Lev a,
    \quad
    \len a = (\Lev a)(1)
    \quad\text{and}\quad
    \content{\Lev a} = \sum_{n=1}^{\len a}n\cdot a(n)
\]
in analogy with~\eqref{eqn-lenGol}
and introduce the \df{$n$-th Levine number}
\begin{gather}\label{eqn-levna}
    \lev_na := \len \Lev^{n-1}a
    \quad\text{for }n\in\N
\end{gather}
as well as the \df{Levine sequence} $\lev a:=(\lev_na)_{n\in\N}$
\df{based at} $a$.
The most prominent example is Levine's original sequence~\eqref{eqn-lev2}.
Until recently, its last verified term has been $\lev_{18}(2)$,
a 137-digit number contributed by Johan Claes in 2008 (see~\OEIS{011784}).
In 2011, Claes also calculated a candidate for $\lev_{19}(2)$,
based on certain modular assumptions.
By means of Theorem~\ref{thm-lev}, we can now confirm this candidate
and calculate the long-sought twentieth term $\lev_{20}(2)$.
Two other Levine sequences
\begin{align*}
    \lev(0,0,1) &= 
    (3, 1, 3, 3, 6, 10, 28, 108, 1011, 32511, 9314238, 84560776390,\dotsc),\\
    \lev(0,2) &=
    (2, 2, 4, 6, 13, 35, 171, 1934, 97151, 52942129, 1435382350480,\dotsc),
\end{align*}
can be found in the OEIS as~\OEIS{061892} and \OEIS{061894}.
To these we will also add more terms.

\section{Discrete calculus for functionals}\label{sec-calculus}

We need to develop some discrete calculus
for $\Q$-valued maps $f\colon\Z^\N\to\Q$, $a\mapsto fa$,
which we call \df{functionals} (although they need not be linear).
Under pointwise addition and multiplication,
they obviously form a commutative ring $(\Q^{\Z^\N},+,\cdot)$.
As that ring's basic constituents we introduce the \df{coordinate functionals}
\[
    x_n \colon \Z^\N \to \Z,
    \quad
    a \mapsto x_n a := a(n)
\]
for $n\in\N$.
From given functionals $f\colon\Z^\N\to\Q$
and $g_1,g_2,g_3,\dotsc\colon\Z^\N\to\Z$
we can form the composed functional
$f(g_1,g_2,g_3,\dotsc)\colon\Z^\N\to\Q$,
$a\mapsto f(g_1a,g_2a,g_3a,\dotsc)$,
in particular, we trivially have $f=f(x_1,x_2,x_3,\dotsc)$.
Let us now introduce the discrete partial \df{derivative} $\dder_1f$
and \df{integral} $\dint_1f$ with respect to $x_1$
of a given functional $f\colon\Z^\N\to\Q$
by setting
\begin{align*}
    (\dder_1f)a &:= f\bigl(a+(1)\bigr) - fa,\\
    (\dint_1f)a &:= \begin{cases}
     \sum_{c=1}^{a(1)}f\bigl(a-(c)\bigr) & \text{if }a(1)>0,\\
     0 & \text{if }a(1)=0,\\
     -\sum_{c=0}^{-a(1)-1}f\bigl(a+(c)\bigr) & \text{if }a(1)<0
    \end{cases}
\end{align*}
for $a \in \Z^\N$.
One could define $\dint_k$ and $\dder_k$
for arbitrary $k\in\N$ in the same way
but in this treatise we will only need the two operators
$\dint_1, \dder_1 \colon \Q^{\Z^\N} \to \Q^{\Z^\N}$
and prove that they are inverse to each other
and linear in the following sense.
\begin{lem}\label{lem-derint}
    For all $f,g\in\Q^{\Z^\N}$,
    the following statements hold.
    \begin{substate}
    \item $\dder_1\dint_1f=f$.\label{sub-derint}
    \item $\dint_1\dder_1f = f - f(0,x_2,x_3,\dotsc)$.\label{sub-intder}
    \item $\dder_1(f\pm g) = \dder_1f \pm \dder_1g$.\label{sub-dersum}
    \item $\dint_1(f\pm g) = \dint_1f \pm \dint_1g$.\label{sub-intsum}
    \item If $\dder_1f=0$, then $\dder_1fg = f\cdot\dder_1g$.
        \label{sub-fderg}
    \item If $\dder_1f=0$, then $\dint_1fg = f\cdot\dint_1g$.
        \label{sub-fintg}
    \end{substate}
\end{lem}
\begin{proof}
Let $a\in\Z^\N$ and $f,g \in \Q^{\Z^\N}$.
\begin{subproof}
\item Set $h:=\dint_1f$.\\
    If $a(1)\geq0$, then
    $(\dder_1h)a = h(a+(1))-ha
     = \sum_{c=1}^{a(1)+1}f(a+(1-c)) - \sum_{c=1}^{a(1)}f(a-(c))
     = f(a+(1-1))=fa$.\\
    If $a(1)<0$, then
    $(\dder_1h)a = h(a+(1))-ha
     = -\sum_{c=0}^{-a(1)-2}f(a+(1+c)) + \sum_{c=0}^{-a(1)-1}f(a+(c))
     = f(a+(0))=fa$.
\item If $a(1)\geq0$, then
    $(\dint_1\dder_1f)a = \sum_{c=1}^{a(1)}(\dder_1f)(a-(c))
      = \sum_{c=1}^{a(1)}(f(a-(c-1))-f(a-(c)))
      = f(a-(1-1))-f(a-(a(1)))
      = fa - f(0,x_2a,x_3a,\dotsc)$.\\
    If $a(1)<0$, then
    $(\dint_1\dder_1f)a = -\sum_{c=0}^{-a(1)-1}(\dder_1f)(a+(c))
      = \sum_{c=0}^{-a(1)-1}(f(a+(c))-f(a+(c+1)))
      = f(a+(0))-f(a+(-a(1)-1+1))
      = fa - f(0,x_2a,x_3a,\dotsc)$.
\item[\bf(c)--(d)] are immediate
    from the definitions of $\dder_1$ and $\dint_1$.
\item[\bf(e)--(f)] The assumption $\dder_1f=0$
    obviously implies $f(a+(c))=fa$ for every $c\in\Z$.
    Hence,
    $(\dder_1fg)a = f(a+(1))\cdot g(a+(1))- fa\cdot ga
                       = fa\cdot(\dder_1g)a$,
    which proves~\ref{sub-fderg}.
    In the same way, distinguishing the three cases $a(1)>0$, $a(1)=0$ and $a(1)<0$
    in the definition of $\dint_1$ proves~\ref{sub-fintg} 
    straightforwardly.
\qedhere
\end{subproof}
\end{proof}
The coordinate functionals $x_1,x_2,x_3,\dotsc$ are algebraically independent
(even over $\C$), and the (multivariate) polynomial ring
$\Q[\bx]=\Q[x_1,x_2,\dotsc]$ is a subring of $\Q^{\Z^\N}$,
where we abbreviate $\bx:=(x_1,x_2,\dotsc)$.
Recall that the multivariate powers
\[
  \bx^r := \prod_{j\in\N}x_j^{r(j)}
  \quad\text{for}\quad
  r \in \N_0^{(\N)}
\]
constitute a $\Q$-basis of $\Q[\bx]$.
We call a functional $f\colon\Z^\N\to\Q$ \df{integer-valued}
if all its values reside in $\Z$.
The integer-valued functionals in $\Q[\bx]$
will be referred to
as (multivariate) \df{integer-valued polynomials}.
To distinguish them among all polynomials,
we make use of traditional binomial expression:
For $k\in\N_0$ and any element $u$ in some
commutative ring~$A$ of characteristic $0$
(that is, the ring morphism $\Z\to A$ is injective),
the \qt{binomial power}
\[
    \binom uk := \prod_{j=1}^k\frac{u-j+1}{j}
               = \frac{u(u-1)\dotsm(u-k+1)}{k!}
\]
resides in the ring $\Q A = \N^{-1}A$ of fractions
with respect to the multiplicative monoid~$\N$
in the sense of~\cite[p.~37]{AtMc}.
We list some useful formulas.
\begin{lem}\label{lem-binom}
    Let $k,l\in\N_0$ and $u,v$ be elements in some commutative ring
    of characteristic~$0$. Then the following identities hold.
    \begin{substate}
    \item $\binom{u+v}{k}
        = \sum\limits_{j=0}^k\binom uj\binom v{k-j}$.
        \label{sub-binomu+vk}
    \item For $m,n\in\N_0$, let
        $\binom{m\times n}k$
        denote the number of left- and right-total relations
        $R$ between $\enum1m$ and $\enum1n$ such that $\card R=k$.
        Then $\binom{u\cdot v}{k}
        = \sum\binom{m\times n}k \binom um\binom vn$,
        where the sum is over all pairs $(m,n)\in\N_0^2$
        with $\max\{m,n\}\leq k\leq mn$.
        \label{sub-binomuvk}
    \item $\binom uk\binom ul
        = \sum\limits_{n=\max\{k,l\}}^{k+l}
          \frac{n!}{(n-k)!(n-l)!(k+l-n)!}\binom un$.
        \label{sub-binomukum}
    \item For $n\in\N_0$, let
        $\threecol nkl := \card\set{\cS}{
            \card\cS=l,\,
            \bigcup_{S\in\cS}S=\enum1n,\,
            \card S=k\text{ for all }S\in\cS}$
        be the number of coverings of $\enum1n$
        by $l$ distinct $k$-subsets.\\
        Then 
        $\displaystyle\binom{\binom uk}l
        = \sum_{n=\min\set{N\in\N_0}{l\leq\binom Nk}}^{k\cdot l}
          \threecol nkl \binom un$.
        \label{sub-binomumk}
    \end{substate}
\end{lem}
\begin{proof}
    Because $\N_0$ is infinite and all four identities
    hold for $u,v\in\N_0$,
    they hold for $u=x_1$, $v=x_2$, 
    in which case
    both sides of each equation
    reside in the polynomial ring $\Q[x_1,x_2]$.
    The identities then follow for arbitrary $u$, $v$
    by substitution.
\end{proof}

For each $r\in \N_0^{(\N)}$,
we can now consider the \df{multivariate binomial power}
\[
    \binom\bx r:=\prod_{j\in\N}\binom{x_j}{r(j)} \in \Q[\bx].
\]
From~\cite[p.~290]{CaCh} and~\cite{Os},
we distil the multivariate version
of a theorem by Pólya~\cite{Po}.

\begin{prp}\label{prp-mivpoly}
The multivariate binomial powers $\binom \bx r$ for $r\in \N_0^{(\N)}$
constitute a $\Q$-basis of $\Q[\bx]$.
The integer-valued polynomials are exactly
the $\Z$-linear combinations of these basis vectors.
\end{prp}

Transformation between the two bases
$(\bx^r)_{r\in\N_0^{(\N)}}$ and $\left(\binom\bx r\right)_{r\in\N_0^{(\N)}}$
is accomplished by means of the Stirling numbers 
\OEIS{008275} and \OEIS{008277}
of the first and second kind along the same lines
as are well-known for univariate polynomials.
We skip the technical details
of explicitly writing this out in our multivariate situation.
Let us call a polynomial $f\in\Q[\bx]$
\df{strongly positive} if its non-zero coefficients
with respect to the basis in Proposition~\ref{prp-mivpoly}
are all positive, that is, $f=\sum_{r\in\N_0^{(\N)}} c_r\binom\bx r$
with $c_r\geq0$ for all $r\in\N_0^{(\N)}$.
Clearly, every such $f$ maps $\N_0^\N$ to $\N_0$,
but $f=1-x_1+\binom{x_1}2$ exemplifies that the converse fails.

\begin{lem}\label{lem-strongpos}
If $f, g_1, g_2,\dotsc\in\Q[\bx]$ are strongly positive,
then so is $f(g_1,g_2,\dotsc)$.\footnote{%
    Even if $g_1,g_2,\dotsc$ are not integer-valued,
    $f(g_1,g_2,\dotsc)$ is still well-defined
    as a functional $\Z^\N\to\Q$.%
}
\end{lem}

\begin{proof}
This follows from Lemma~\ref{lem-binom}.
\end{proof}

\begin{lem}\label{lem-poly}
    Let $f\in\Q[\bx]$.
    The following statements hold.
    \begin{substate}
    \item $\dder_11=0$, and
        $\dder_1\binom{x_1}k = \binom{x_1}{k-1}$
        for every $k\in\N$.\label{sub-derbinom}
    \item $\dint_1\binom{x_1}k = \binom{x_1}{k+1}$
        for every $k\in\N_0$,
        in particular $\dint_11 = x_1$.\label{sub-intbinom}
    \item $\dder_1f \in \Q[\bx]$.\label{sub-derf}
    \item $\dint_1f \in x_1\Q[\bx]$.\label{sub-intf}
    \item $\dder_1f=0 \iff f\in\Q[x_2,x_3,\dotsc]$.\label{sub-const}
    \end{substate}
\end{lem}
\begin{proof}
    Statements~\ref{sub-derbinom} and~\ref{sub-intbinom}
    follow from well-known properties of the usual binomial coefficients.
    As for the other assertions,
    write $f=\sum_{k=0}^d\binom{x_1}k f_k$
    with $d\in\N_0$ and $\thru f0d\in\Q[x_2,x_3,\dotsc]$.
    Then
    $\dint_1f = \sum_{k=0}^d \binom{x_1}{k+1} f_k$
    and
    $\dder_1f = \sum_{k=1}^d \binom{x_1}{k-1} f_k$
    by~\ref{sub-derbinom} and~\ref{sub-intbinom}.
\end{proof}

\section{Operators on the free group over the integers}\label{sec-FZ}

In order to unveil the algebraic structure
behind the golombic and Levine sequences,
we allow for arbitrary integer terms
and also for arbitrary integer multiplicities.
We denote by $(\FZ,\cat)$ the free group\footnote{%
    As in~\cite[p.~87]{Mc},
    $F\colon\Set\to\Grp$ is the left adjoint
    for the forgetful functor $\Grp\to\Set$.%
}
over $\Z$
and write its elements,
which we call \df{words} (by slight abuse of language),
as parenthesised finite formal power products,
so that $()$ denotes the empty word,
the identity element in $\FZ$.
The group operation $\cat$ is concatenation of words
subject only to the reduction rules
\begin{gather}\label{eqn-redrules}
    (b^m)\cat(b^n)=(b^mb^n)=(b^{m+n})
    \quad\text{and}\quad
    (b^0)=()
    \quad\text{for }b,m,n\in\Z,
\end{gather}
that is, adjacent formal powers with equal base \qt{melt together}
and may even disappear.
This leads to the following well-known fact.
\begin{rmk}\label{rmk-FZreduced}
    Every word has a unique representation
    $(b_1^{m_1}\!\dotsc b_k^{m_k})$ subject to the conditions
    $k\in\N_0$,
    $\thru m1k\in\Z\sm\{0\}$, $\thru b1k\in\Z$
    and $b_j\neq b_{j+1}$ for $1\leq j<k$.
\end{rmk}
The representation in Remark~\ref{rmk-FZreduced}
will be referred to as \df{reduced}.
A formal power with base $b\in\Z$ and exponent $1$
will sometimes be abbreviated $b^1=b$.
The inverse of a word $w$ in the group $(\FZ,\cat)$ will be denoted $w^{-1}$.
We consider the {\em set}\, $\Z^{(\N)}$ as a submonoid
$(\Z^{(\N)},\cat)$ of $(\FZ,\cat)$ in virtue of the embedding
\[
    \Z^{(\N)}\into\FZ,\quad a\mapsto (a(1)^1\!\dotsc a(\len a)^1)
\]
and extend the length and content maps
$\len{}$ and $\content{\ }$
introduced in Section~\ref{sec-intseq}
to  group homomorphisms\footnote{%
    In the natural bijection $\Grp(\FZ,\Z)\simeq\Set(\Z,\Z)$
    of hom-sets (cf.~\cite[pp.~79--82]{Mc}),
    $\len{}$ and $\content{\ }$ are the left adjuncts
    of the constant function $\Z\to\Z$, $z\mapsto1$
    and of the identity map $\id_\Z$, respectively.%
}
$(\FZ,\cat)\to(\Z,+)$ by setting
\[
    \len{(b_1^{m_1}\!\dotsc b_k^{m_k})} := m_1+\dotsm+m_k
    \quad\text{and}\quad
    \content{(b_1^{m_1}\!\dotsc b_k^{m_k})} := m_1b_1+\dotsm+m_kb_k
\]    
for $k\in\N_0$, $\thru b1k,\thru m1k\in\Z$.
Both these homomorphisms are compatible
with the reduction rules~\eqref{eqn-redrules},
which is why we did not require their arguments to be reduced.
The next theorem establishes for every $z\in\Z$
an operator $\Gol_z\colon\FZ\to\FZ$, $w\mapsto \Gol_zw$,
which will be seen to naturally generalise the Golombic operator
$\Gol\colon \N_0^{(\N)}\to \N_0^{(\N)}$ introduced in Section~\ref{sec-intseq}.

\begin{thm}\label{thm-Golz}
    There is a unique family $(\Gol_z)_{z\in\Z}$ of maps
    $\Gol_z\colon\FZ\to\FZ$ satisfying the three axioms
    \begin{align}
	    \label{eqn-Golz0}
	    \Gol_z() &= ()
	    &&\text{for all } z\in\Z,\\
	    \label{eqn-Golzb}
	    \Gol_z(b) &= (z^b)
	    &&\text{for all } z,b\in\Z,\\
	    \label{eqn-Golzvw}
	    \Gol_z(v\cat w) &= \Gol_zv \cat \Gol_{z+\len{v}}w
	    &&\text{for all } z\in\Z,\; v,w\in\FZ.
    \end{align}
    
\end{thm}

\begin{proof}
From the three axioms, we infer
$()=\Gol_z(b^{\mi1}b^1)=\Gol_z(b^{\mi1})\cat\Gol_{z-1}(b)$,
hence
\[
    \Gol_z(b^{\mi1}) = (\Gol_{z-1}(b))^{-1}
                     = ((z-1)^b)^{-1} = ((z-1)^{\mi b})
    \quad \text{for all } b,z\in\Z.
\]
Using induction and all three axioms once more
leads to the three rules
\begin{gather}
    \label{eqn-Golzbm}
    \Gol_z(b^m) = (z^b\!\dotsc(z+m-1)^b),\;
    \Gol_z(b^{\mi m}) = ((z-1)^{\mi b}\!\dotsc(z-m)^{\mi b})
    \;\text{ for } b\in\Z, m\in\N,\\
    \label{eqn-Golzbm1k}
    \Gol_zw = \Gol_{z+s_0}(b_1^{m_1}) \cat\dotsm\cat \Gol_{z+s_{k-1}}(b_k^{m_k})
    \quad\text{for any reduced } w = (b_1^{m_1}\!\dotsc b_k^{m_k}) \in \FZ,\\
    \nonumber
    \text{where } s_0=\len()=0,\; s_j := \len(b_1^{m_1}\!\dotsc b_{j}^{m_{j}})
                       = m_1+\dotsm+m_j
    \text{ for } j\in\enum1k,
\end{gather}
for any given $z\in\Z$.
According to Remark~\ref{rmk-FZreduced},
these rules define a family $(\Gol_z)_{z\in\Z}$
of maps $\Gol_z\colon\FZ\to\FZ$.
Axiom~\eqref{eqn-Golz0} then follows from rule~\eqref{eqn-Golzbm1k} for $w=()$,
and Axiom~\eqref{eqn-Golzb} is the first rule in~\eqref{eqn-Golzbm} with $m=1$.
All we need to still verify is
that this family satisfies Axiom~\eqref{eqn-Golzvw},
which we first prove a special version of, namely
\begin{gather}
    \label{eqn-Golzbm+n}
    \Gol_z(b^{m+n}) = \Gol_z(b^m)\cat \Gol_{z+m}(b^n)
    \quad \text{for all } z,b,m,n\in\Z.
\end{gather}
Let $z\in\Z$.
For $mn\geq0$,
\eqref{eqn-Golzbm+n} follows directly from~\eqref{eqn-Golzbm},
either the first or the second rule.
The other cases require combining both rules:
If $m<0=m+n$, then
\[
    \Gol_z(b^{m+n}) = ()
    = ((z-1)^{\mi b}\!\dotsc(z+m)^{\mi b}) \cat ((z+m)^b\!\dotsc(z-1)^b)
    = \Gol_z(b^m)\cat \Gol_{z+m}(b^n)
\]
by~\eqref{eqn-Golz0}, \eqref{eqn-redrules} and~\eqref{eqn-Golzbm},
and likewise for $n<0=m+n$.
For $m<0<m+n$ we reason
\begin{align*}
    \Gol_z(b^{m+n})
    &= (z^b\!\dotsc(z+m+n-1)^b)\\
    &= ((z-1)^{\mi b}\!\dotsc(z+m)^{\mi b})
       \cat((z+m)^b\!\dotsc(z-1)^bz^b\!\dotsc(z+m+n-1)^b)\\
    &= \Gol_z(b^m)\cat \Gol_{z+m}(b^n).
\end{align*}
There are three more cases,
namely\; $m<m+n<0$,\; $n<0<m+n$,\; $n<m+n<0$,\;
and they are settled similarly,
thus we may consider~\eqref{eqn-Golzbm+n} proved.

To verify~\eqref{eqn-Golzvw},
let $v=(a_1^{m_1}\!\dotsc a_k^{m_k})$,
$w=(b_1^{n_1}\!\dotsc b_l^{n_l})\in\FZ$
both be in reduced representation and $z\in\Z$.
If $kl=0$, then at least one of $v$, $w$ is the empty word
and~\eqref{eqn-Golzvw} holds trivially by~\eqref{eqn-Golz0}.
We may therefore assume $kl>0$
and proceed by induction on $k+l$.
If $a_k\neq b_1$,
then~\eqref{eqn-Golzbm1k} entails~\eqref{eqn-Golzvw}
straightforwardly.
If $a_k = b_1$, then
\[
    v' := v \cat (b_1^{n_1}) = (a_1^{m_1}\!\dotsc a_{k-1}^{m_{k-1}}a_k^{m_k+n_1})
    \quad\text{and}\quad
    w' := (b_1^{\mi n_1})\cat w = (b_2^{n_2}\!\dotsc b_l^{n_l})
\]
satisfy $v\cat w = v'\cat w'$ and
$
    \Gol_zv'
    = \Gol_z(a_1^{m_1}\!\dotsc a_{k-1}^{m_{k-1}}) \cat \Gol_{z+\len{v}-m_k}(a_k^{m_k+n_1})
    = \Gol_zv \cat \Gol_{z+\len{v}}(b_1^{n_1})
$
by~\eqref{eqn-Golzbm1k} and~\eqref{eqn-Golzbm+n},
hence, applying the induction hypothesis to $v'\cat w'$, yields
\[
    \Gol_z(v\cat w) = \Gol_z(v'\cat w') = \Gol_zv' \cat \Gol_{z+\len{v'}}w'
    = \Gol_zv \cat \Gol_{z+\len{v}}(b_1^{n_1}) \cat \Gol_{z+\len{v}+n_1}w'
    = \Gol_zv \cat \Gol_{z+\len{v}}w
\]
using~\eqref{eqn-Golzbm1k} again.
\end{proof}

Having proved the theorem, we can now, of course,
drop the assumption in formula~\eqref{eqn-Golzbm1k}
that the word $w$ be given in reduced form,
which we had only included for technical reasons.
We list some immediate consequences.

\begin{cor}\label{cor-Golw}
Let $z,m,b\in\Z$ and $w\in\FZ$. Then
\begin{substate}
\item $\Gol_z(0^m)=()$,\label{sub-Golz0m}
\item $\Gol_1\rest{\N_0^{(\N)}} = \Gol$,\label{sub-Gol1}
\item $\len{\Gol_zw} = \content w$,\label{sub-lenGolzw}
\item $\content{\Gol_z(b^m)} = mb(z+\frac{m-1}2)$,\label{sub-contGolzbm}
\item $\content{\Gol_zw} = \content{\Gol_0w} + z\content{w}$,\label{sub-contGolzw}
\item $(\Gol_zw)^{-1} = \Gol_{z+\len{w}}(w^{-1})$.\label{sub-Golw-1}
\end{substate}
\end{cor}

\begin{proof}
\begin{subproof}
\item Use~\eqref{eqn-Golz0} for $m=0$ and~\eqref{eqn-Golzbm} for $m\neq0$.
\item holds due to Axioms~\eqref{eqn-Golz0}--\eqref{eqn-Golzvw}
        and the definition of $\Gol$ in Section~\ref{sec-intseq}.
\item follows from~\eqref{eqn-Golzbm}, \eqref{eqn-Golzvw}
    and the definition of $\len{}$ and $\content{\ }$.
\item is trivial by Axiom~\eqref{eqn-Golz0} if $m=0$.
  Otherwise apply $\content\ $ to~\eqref{eqn-Golzbm}.
\item follows from~\eqref{eqn-Golzbm1k}, \ref{sub-contGolzbm}
    and the definition of $\content w$.
\item holds because
    $\Gol_zw \cat \Gol_{z+\len{w}}(w^{-1}) = \Gol_z(w\cat w^{-1}) = \Gol_z() = ()$
    by~\eqref{eqn-Golzvw} and~\eqref{eqn-Golz0}.\qedhere
\end{subproof}
\end{proof}

The following observations are remarkable
but will not be used in the sequel.

\begin{rmk}
Let  $z\in\Z$.
\begin{substate}
\item The operator $\Gol_z\colon\FZ\to\FZ$ is surjective but not injective.
\item The restriction $\Gol_z\rest{\Z^{(\N)}}$ is injective with image
    \[
    \Gol_z\Z^{(\N)} = \set{ (b_1^{m_1}\!\dotsc b_k^{m_k}) }
        { k\in\N_0,\;z\leq b_1<\dotsm<b_k,\;\thru m1k\in\Z\sm\{0\} }
    \subsetneq \FZ,
    \]
    and
    $\Gol_z\N_0^{(\N)} = \set{a\in \N_0^{(\N)}}{z\leq a(1)\leq\dotsm\leq a(\len a)}$.
\end{substate}
\end{rmk}

\begin{proof}
\begin{subproof}
\item We easily verify that
    $\Gol_z(0^{b_1-z} m_1^1 0^{b_2-b_1-1} m_2^1 \dotsc0^{b_k-b_{k-1}-1} m_k^1) = w$
    for any given $w=(b_1^{m_1}\!\dotsc b_k^{m_k})\in\FZ$
    due to the rules~\eqref{eqn-Golzbm}--\eqref{eqn-Golzbm1k}.
    On the other hand,
    \[
        (b^{m+n})
        = \Gol_z(0^{b-z}m^10^{\mi1}n^1)
        = \Gol_z(0^{b-z}m^1{\mi n}^{\mi1})
        = \Gol_z(0^{b+1-z}{\mi m}^{\mi1}n^1)
        = \Gol_z(0^{b+1-z}{\mi m}^{\mi1}0^1{\mi n}^{\mi1})
    \]
    for every choice of $b,m,n\in\Z$.
\item Obviously,
    $\Gol_za = (z^{a(1)}\!\dotsc(z+\len a-1)^{a(\len a)})$
    determines $a\in \Z^{(\N)}$  uniquely.
    The latter identity also shows that both images look as asserted.
\qedhere
\end{subproof}
\end{proof}

In view of~\ref{cor-Golw}\ref{sub-Gol1},
$\Gol_z$ generalises the golombic operator $\Gol$
introduced in Section~\ref{sec-intseq}.
As before, we want to investigate iterated compositions
\[
    \Gol_a^n := \Gol_{a(n)}\!\circ\dotsm\circ \Gol_{a(1)} \colon \FZ\to\FZ,\\
\]
for $a\in\Z^\N$, $n\in\N_0$.
(In particular, $\Gol_a^0=\id_\FZ$.)
Let us start by recording a useful formula.

\begin{lem}\label{lem-Golanvw}
    Let $a\in\Z^\N$, $v,w\in\FZ$,
    and set $v_n:=\Gol_a^nv$ for $n\in\N_0$.
    Then
    \[
        \Gol_a^n(v\cat w)
        = v_n \cat \Gol_{a+(\len{v_0},\len{v_1},\len{v_2},\dotsc)}^nw
    \]
    holds for every $n\in\N_0$.
\end{lem}

\begin{proof}
This follows from Axiom~\eqref{eqn-Golzvw} by induction on $n$.
\end{proof}

Vardi~\cite{Va} pictures the iteration $\Gol_a^n$ for $a\in\N_0^\N$
in a \qt{royal family} tree.
He introduces
functionals $T_1,T_2,T_3\dotsc\colon\N_0^\N\to\N_0$
to count the number of children at each generation
and establishes certain polynomial formulas for them.
In our terminology, his definition reads 
$
    T_na := \content{\Gol_a^n(1)}
$
for $n\in\N$, $a\in\N_0^\N$.
Our aim is to extend Vardi's definition
to arbitrary integer sequences $a$
and prove that his formulas still hold.
We thus define the functionals $T_0, T_1,T_2,\dotsc\colon\Z^\N\to\Z$
by
\[
    T_na := \content{\Gol_a^n(1)}
    \quad\text{for } n\in\N_0,\; a\in\Z^\N
\]
and start by listing some properties.
In order to conveniently formulate them,
we abbreviate $f\rest k:=f(\thru x1k,0,0,\dotsc)$
for any functional $f\colon\Z^\N\to\Q$ and $k\in\N_0$.

\begin{rmk}\label{rmk-Tn}
    Let $n\in\N$. We have the following identities.
    \begin{substate}
    \item $T_0 = 1$, $T_1 = x_1$, $T_2 = x_1x_2$, and
        $T_3 = \textstyle x_1x_2 ( \frac{x_1-1}2 + x_3 )
             = x_1x_2x_3 + \binom{x_1}2 x_2$.\label{sub-T0123}
    \item $T_n = T_n\rest n$.\label{sub-Tnrestn}
    \item $T_n(0,x_2,x_3,\dotsc)=0$.\label{sub-Tn0x2}
    \item If $n\geq2$, then $T_n(x_1,0,x_3,x_4,\dotsc)=0$.\label{sub-Tnx10}
    \item $T_n(1,x_2,x_3,\dotsc) = T_{n-1}(x_2,x_3,\dotsc)$.\label{sub-Tn1x2}
    \item $(\dder_1T_n)(0,\thru x2n) = T_{n-1}(x_2,\dotsc,x_n)$.\label{sub-derTn0x2}
    \item $T_n = T_n\rest{n-1} + x_nT_{n-1}$.\label{sub-Tnrest}
    \end{substate}
\end{rmk}

\begin{proof}
We obtain~\ref{sub-T0123} by direct calculation
and~\ref{sub-Tnrestn} from our definition.
The next two assertions follow from~\ref{cor-Golw}\ref{sub-Golz0m}.
Identity~\ref{sub-Tn1x2} holds because $\Gol_1(1)=(1)$,
and~\ref{sub-derTn0x2} follows
from~\ref{sub-Tn0x2} and~\ref{sub-Tn1x2}.
As for~\ref{sub-Tnrest}, let $a\in\Z^\N$ and
apply~\ref{cor-Golw}\ref{sub-contGolzw}
to $z:=a(n)$ and $w:=\Gol_a^{n-1}(1)$.
\end{proof}

In order to re-establish Proposition~3 from~\cite{Va},
which Vardi used to calculate $\thru T14$,
we introduce the map
$\hat{\ }\colon\Z^\N\to\Z^\N$,
$a\mapsto \hat a:= a + (-2a(1),0,T_1a,T_2a,T_3a,\dotsc)$,
that is,
\[
    \hat a(1) = -a(1),\quad
    \hat a(2) = a(2)
    \quad\text{and}\quad
    \hat a(n+2) = a(n+2)+T_na
    \quad\text{for } n\in\N.
\]

As an aside, we record a functional equation.

\begin{prp}\label{prp-Tnhata}
    Let $a\in\Z^\N$ and $n\in\N$.
    Then $T_n\hat a = -T_na$.
\end{prp}

\begin{proof}
Abbreviating $w_n := \Gol_a^n(1)$,
we obtain
\begin{gather}\label{eqn-hatawn}
    \hat a(n+2) = a(n+2)+T_na = a(n+2)+\content{w_n}
                = a(n+2)+\len{w_{n+1}}
\end{gather}
by~\ref{cor-Golw}\ref{sub-lenGolzw}.
The assertion follows for $n\leq2$
from applying the group homomorphism $\content{\ }$
to the equations $\Gol_{\hat a}^1(1) = \Gol_{-a(1)}(1) = (-a(1))$ and 
$\Gol_{\hat a}^2(1) = \Gol_{a(2)}(-a(1)) = (a(2)^{\mi a(1)}) = w_2^{-1}$,
the second of which moreover entails
\[
    () = \Gol_{\Ver2a}^{n-2}(w_2\cat w_2^{-1})
    = \Gol_{\Ver2a}^{n-2}w_2
      \cat \Gol_{\Ver2\hat a}w_2^{-1}
    = \Gol_a^n(1) \cat \Gol_{\hat a}^n(1)
\]
for $n\geq3$ by~\eqref{eqn-Golz0}, \eqref{eqn-hatawn} and Lemma~\ref{lem-Golanvw}.
Applying $\content{\ }$ here as well completes the proof.
\end{proof}

\begin{cor}\label{cor-hathat}
The map $\hat{\ }\colon\Z^\N\to\Z^\N$ is an involution,
that is, $\hathat a=a$ for every $a\in\Z^\N$.
\end{cor}

\begin{proof}
Let $a\in\Z^\N$ and $n\in\N$.
Then $\hathat a(1) = -\hat a(1) = a(1)$,
$\hathat a(2) = \hat a(2) = a(2)$ and
$
    \hathat a(n+2) = \hat a(n+2)+T_n\hat a = a(n+2)+T_na-T_na = a(n+2)
$
by~\ref{prp-Tnhata}.
Hence,~$\hathat a = a$.
\end{proof}

After our small detour, will now re-establish Vardi's Proposition~3
from page~8 of~\cite{Va}.

\begin{prp}\label{prp-intTn}
For every $n\in\N$, we have
$T_n = \dint_1\bigl(T_{n-1}(x_2,x_3+T_1,x_4+T_2,\dotsc)\bigr)$.
\end{prp}

\begin{proof}
The case $n=1$ follows directly from~\ref{rmk-Tn}\ref{sub-T0123}
and~\ref{lem-poly}\ref{sub-intbinom}.
Let $a\in\Z^\N$.
As in our proof of Proposition~\ref{prp-Tnhata},
we set $w_n := \Gol_a^n(1)$ for $n\in\N$.
From
$\Gol^2_{a+(1)}(1) = w_2\cat(a(2))$,
Lemma~\ref{lem-Golanvw} and~\eqref{eqn-hatawn}
we conclude for $n\geq2$ that
\[
    \Gol_{a+(1)}^n(1)
    = \Gol_{\Ver2a}^{n-2}(w_2\cat(a(2)))
    = \Gol_{\Ver2a}^{n-2}w_2\;\!
      \cat \Gol_{\Ver2a+(\len{w_2},\len{w_3},\dotsc)}^{n-2}(a(2))
    = \Gol_a^n(1) \cat \Gol_{\Ver{}\hat a}^{n-1}(1),
\]
thus
$(\dder_1T_n)a
 = \content{\Gol_{\Ver{}\hat a}^{n-1}(1)}
 = T_{n-1}\Ver{}\hat a$, that is,
$\dder_1T_n=T_{n-1}(x_2,x_3+T_1,x_4+T_2,\dotsc)$.
The assertion now follows with~\ref{lem-derint}\ref{sub-intder}
and~\ref{rmk-Tn}\ref{sub-Tn0x2}.
\end{proof}

This proposition
reveals that Vardi's functionals are in fact polynomials,
even after our extending their domain to arbitrary integers.
We have implemented the recursive formula
of Proposition~\ref{prp-intTn} in PARI
and calculated Vardi's polynomials $\thru T19$
within less than 17 hours.
Together they take up around 150 megabytes of memory.
To calculate $T_{10}$ would cost
an estimated 14,000 hours and over 4~gigabytes of memory.
Although we know of some tricks to cut time and memory in half,
Vardi's tenth polynomial is still out of reach.

\begin{cor}\label{cor-Tn}
    Let $n\in\N$.
    Vardi's functional
    $T_n$ is a strongly positive integer-valued polynomial
    residing in the ideal $x_1\Q[\thru x1n]$.
    Moreover, $T_n\rest{n-1}\in x_1x_2\Q[\thru x1{n-1}]$.
\end{cor}

\begin{proof}
The functional $T_n$ is a polynomial
by~\ref{rmk-Tn}\ref{sub-T0123}, Proposition~\ref{prp-intTn} and induction.
It is integer-valued by definition
and resides in $x_1\Q[\thru x1n]$
by~\ref{rmk-Tn}\ref{sub-Tnrestn} and~\ref{sub-Tn0x2}.
Its strong positivity then follows by induction
from~\ref{rmk-Tn}\ref{sub-T0123}, Proposition~\ref{prp-intTn},
Lemmas~\ref{lem-strongpos},
\ref{lem-poly}\ref{sub-intbinom}, \ref{sub-const}
and \ref{lem-derint}\ref{sub-fintg}.
The last statement is a consequence
of Remark~\ref{rmk-Tn}\ref{sub-T0123},
          \ref{sub-Tn0x2} and \ref{sub-Tnx10}.
\end{proof}

Corollary~\ref{cor-Golw}\ref{sub-Gol1} entitles us
to abbreviate $\Gol:=\Gol_1\colon\FZ\to\FZ$
thus extending the golombic operator to words.
For $w\in\FZ$, we set
$\gol w := (\gol_nw)_{n\in\N} \in \Z^\N$
with
$\gol_nw := \len{\Gol^{n-1}w}$
for $n\in\N$
and thereby generalise the notions of
golombic sequence and golombic number
introduced in Section~\ref{sec-intseq} from tuples to words.
The golombic numbers based at an arbitrary given word
can be calculated by means of the following iterative formula.

\begin{thm}\label{thm-gol}
    We have
    $\gol_n(v\cat(b^m))
     = \gol_nv + T_n(m,b,1+\gol_1v,1+\gol_2v,\dotsc)$
    for every $n\in\N$, $v\in\FZ$, $b,m\in\Z$.
\end{thm}

\begin{proof}
Let $v\in\FZ$ and $a=(m,b,1+\gol_1v,1+\gol_2v,\dotsc)$.
We note that $v_n:=\Gol^nv$ satisfies $\len{v_n}=\gol_{n+1}v$
for every $n\in\N_0$.
For $2\leq n\in\N$ and $b,m\in\Z$,
we may therefore conclude
\[
    \gol_n(v\cat(b^m))
    = \content{\Gol^{n-2}(v \cat \Gol_{m,b}(1))}
    = \content{\Gol^{n-2}v} + \content{\Gol_a^n(1)}
    = \gol_nv + T_na
\]
by Lemma~\ref{lem-Golanvw} and Corollary~\ref{cor-Golw}\ref{sub-lenGolzw}.
The case $n=1$ can be verified directly.
\end{proof}

\begin{figure}[!ht]\centering
\renewcommand{\figurename}{Algorithm}
\begin{struktogramm}(128,98)[{\tt golombic\_numbers}$(b_1^{m_1}\!\dotsc b_k^{m_k})$]
\assign{$s \ot (1^9)$\hfill{\em\# nine ones}}
\forallin{{\tt for\ }$n \ot 1,\dotsc,9${\tt\ do}\hfill{\em\# we only have $\thru T19$}}
\assign{\parbox{16em}{Pick $d_n\in\N$ minimal with\\
              $\ti T_n \ot \frac{d_n}{x_1\cdot x_2}\cdot T_n\rest{n-1}\in\Z[\bx]$}}
\forallinend
\forallin{{\tt for\ }$j \ot 1,\dotsc,k${\tt\ do}}
\assign{$\delta \ot 1$}
\assign{$a \ot (m_j,b_j)\cat s$}
\forallin{{\tt for\ }$n \ot 1,\dotsc,9${\tt\ do}}
\assign{$\delta \ot \frac{m_jb_j\cdot\ti T_na}{d_n} + a(n)\cdot\delta$}
\assert{At this point we have $\delta = T_na$.}
\assign{$s(n) \ot  s(n) + \delta$}
\assert{Now we have $s(n) = 1+\gol_n(b_1^{m_1}\!\dotsc b_j^{m_j})$.}
\forallinend
\forallinend
\return{{\tt return }$s - (1^9)$}
\end{struktogramm}
\caption{Calculate $\gol_1w,\dotsc,\gol_9w$
         from $w=(b_1^{m_1}\!\dotsc b_k^{m_k})\in\FZ$.}
\label{alg-gol}
\end{figure}

Theorem~\ref{thm-gol} leads us to Algorithm~\ref{alg-gol}.
In order to reduce execution time,
we have applied a few optimisations:
We split up each Vardi polynomial $T_n$
according to~\ref{rmk-Tn}\ref{sub-Tnrest},
extract its common denominator $d_n$, divide by $x_1x_2$
(possible by Corollary~\ref{cor-Tn})
and end up with a primitive polynomial
$\ti T_n := \frac{d_n}{x_1x_2}\cdot T_n\rest{n-1} \in \Z[\thru x1{n-1}]$.
For example,  $d_1=d_2=1$, $\ti T_1 = \ti T_2 = 0$, $d_3=2$
and $\ti T_3=x_1-1$ by~\ref{rmk-Tn}\ref{sub-T0123}.
These decompositions are reassembled in the inner-most loop.
We thereby avoid time-costly calculations with rationals
and earn a speed-up by a factor of about 8 in total.
The two assertions in oval boxes are meant to assist the reader
with verifying that the algorithm still produces
the results as specified and expected.
In Section~\ref{sec-compute}
we will use Algorithm~\ref{alg-gol}
to calculate the first 19 terms of~$\gol(2)=$~\OEIS{014644}.

The functional equation~\ref{prp-Tnhata}
is tightly connected with the anti-involution $w\mapsto w^{-1}$ on $\FZ$.
More functional equations can be gained from the involutions
$-$ and $\bar{\ }\colon\FZ\to\FZ$ given by
\[
  -w := (\mi b_1^{m_1}\!\dotsc\mi b_k^{m_k})
  \quad\text{and}\quad
  \bar w := (b_1^{\mi m_1}\!\dotsc b_k^{\mi m_k})
  \quad\text{for }
  w = (b_1^{m_1}\!\dotsc b_k^{m_k})\in\FZ.
\]
\begin{rmk}\label{rmk-bar}
    The above involutions are group automorphisms of $(\FZ,\cat)$
    and commute with each other, that is,
    \begin{gather*}    
        -(v\cat w)=-v\cat(-w),
        \quad
        \ov{v\cat w}=\bar v\cat \bar w,
        \quad
        \ov{-w}=-\bar w,
        \quad\text{moreover}\\
        \len{-w}=\len w,
        \quad
        \content{-w}=-\content w,
        \quad
        \len{\bar w} = -\len w,
        \quad
        \content{\bar w}=-\content w
    \end{gather*}
    for $v,w\in\FZ$.
\end{rmk}
\begin{proof}
All these identities are obvious.
\end{proof}

\begin{lem}\label{lem-Golinvo}
    Let $z\in\Z$ and $w\in\FZ$.
    Then $\Gol_z(-w) = \ov{\Gol_zw} = -\Gol_{1-z}\bar w$.
\end{lem}

\begin{proof}
    Both equalities follow
    from rules~\eqref{eqn-Golzbm}--\eqref{eqn-Golzbm1k}.
\end{proof}

This lemma leads to three functional equations
\begin{align*}
T_n &= (-1)^{n^2\bmod3}
      T_n(-x_1,x_2,1-x_3,1-x_4,x_5,1-x_6,1-x_7,x_8,1-x_9,\dotsc),\\
T_n &= (-1)^{(n-1)^2\bmod3}
      T_n(x_1,-x_2,x_3,1-x_4,1-x_5,x_6,1-x_7,1-x_8,x_9,\dotsc),\\
T_n &= (-1)^{(n+1)^2\bmod3}
      T_n(-x_1,-x_2,1-x_3,x_4,1-x_5,1-x_6,x_7,1-x_8,1-x_9,\dotsc)
\end{align*}
for each $n\in\N$.
When combined with Proposition~\ref{prp-Tnhata},
they entail a total of eight functional equations
(including the trivial equation $T_n=T_n$),
which involve an obvious action of the group
$(\Z/2\Z)^3$ on the arguments of $T_n$,
that is, on the set $\Z^\N$.
This group action reflects the action on~$\FZ$
by its group of (anti-)involutions,
which is generated by the two involutions $-$,~$\bar{\ }$
and the anti-involution ${}^{-1}\colon\FZ\to\FZ$
and is isomorphic to $(\Z/2\Z)^3$ as well.
We will not make use of all this and leave
the explicit statements and proofs as an exercise.
Let us now extend the anti-involution $\rev{\ }$
introduced on the monoid $(\Z^{(\N)},\cat)$ in Section~\ref{sec-intseq}
to the group $\FZ$
by setting
\[
    \rev w := \bar w^{-1} = (b_k^{m_k}\!\dotsc b_1^{m_1})
    \quad \text{for} \quad
    w = (b_1^{m_1}\!\dotsc b_k^{m_k}) \in\FZ,
\]
and write down its elementary properties
(which follow from its definition or from~\ref{rmk-bar}):

\begin{rmk}\label{rmk-rev}
    For $v,w\in\FZ$
    we have
    $\longrev{v\cat w} = \rev w \cat \rev v$,
    $\len{\rev w} = \len w$ and
    $\content{\rev w} = \content w$.
\end{rmk}

For given $z\in\Z$, we define the operator
$\Lev_z \colon \FZ \to \FZ$, $w \mapsto \longrev{\Gol_zw}$.
Then $\Lev_1\rest{\N_0^{(\N)}} = \Lev$,
so that we may extend the Levine operator from Section~\ref{sec-intseq}
by setting
$L:=L_1\colon\FZ\to\FZ$ and call
\[
    \lev_nw := \len{\Lev^{n-1}\!w}
    \text{ for } n\in\N \quad\text{and}\quad
    \lev w := (\lev_nw)_{n\in\N}
\]
the \df{$n$-th Levine number} and the \df{Levine sequence}
based at the word $w$,
in natural generalisation of definition~\eqref{eqn-levna}.
The following theorem presents a recipe for calculating
these Levine numbers.

\begin{thm}\label{thm-lev}
    Let $w=(b_1^{m_1}\dots b_k^{m_k}) \in\FZ$ with $k\in\N_0$ and
    $\thru m1k, \thru b1k \in \Z$.
    For $n\in\N$ and $j\in\enum1k$, we recursively  calculate
    the integers
    \[
        l_{n,0} := 0,\quad
        l_{n,j} := l_{n,j-1} + (-1)^n
        T_n( -m_j, -b_j, -l_{1,j-1}, l_{2,j-1}-l_{2,k},
                         -l_{3,j-1}, l_{4,j-1}-l_{4,k},\dotsc).
    \]
    Then $\lev_nw = l_{n,k}$ for every $n\in\N$.
\end{thm}

Theorem~\ref{thm-lev} gives rise to Algorithm~\ref{alg-lev}
which calculates the first nine Levine numbers of a given word
at the price of its second, fourth and sixth Levine number.
Other than that, we use the same decompositions
and reassembly for our speed-up as in the golombic case.
Note that the tuple variable $s$ holds the numbers $l_{n,j}$
multiplied with the extra sign factor $(-1)^n$,
which is stripped off when returning the result.
Again, the two assertions in oval boxes
are meant to aid the reader in verifying the algorithm.
In Section~\ref{sec-compute} we will use Algorithm~\ref{alg-lev}
to calculate the first 20, 19 and 18 terms
of the three Levine sequences
$\lev(2)$, $\lev(0,0,1)$ and $\lev(0,2)$,
respectively.
\newpage

\begin{figure}[!ht]\centering
\renewcommand{\figurename}{Algorithm}
\begin{struktogramm}(128,103)[{\tt levine\_numbers}$(b_1^{m_1}\!\dotsc b_k^{m_k})$]
\assign{$h \ot (0,\lev_2w,0,\lev_4w,0,\lev_6w)$}
\assign{$s \ot ()$}
\forallin{{\tt for\ }$n \ot 1,\dotsc,9${\tt\ do}\hfill{\em\# we only have $\thru T19$}}
\assign{\parbox{16em}{Pick $d_n\in\N$ minimal with\\
              $\ti T_n \ot \frac{d_n}{x_1\cdot x_2}\cdot T_n\rest{n-1}\in\Z[\bx]$}}
\forallinend
\forallin{{\tt for\ }$j \ot 1,\dotsc,k${\tt\ do}}
\assign{$\delta \ot 1$}
\assign{$a \ot (-m_j,-b_j) \cat (s-h)$}
\forallin{{\tt for\ }$n \ot 1,\dotsc,9${\tt\ do}}
\assign{$\delta \ot \frac{m_jb_j\cdot\ti T_na}{d_n} + a(n)\cdot\delta$}
\assert{At this point we have $\delta = T_na$.}
\assign{$s(n) \ot  s(n) + \delta$}
\assert{Now we have $s(n) = (-1)^n l_{n,j}$.}
\forallinend
\forallinend
\return{{\tt return }$((-1)^ns(n))_{n=1,\dotsc,9}$}
\end{struktogramm}
\caption{Calculate $\lev_1w,\dotsc,\lev_9w$
         from $w=(b_1^{m_1}\!\dotsc b_k^{m_k})\in\FZ$
         and $\lev_2w$, $\lev_4w$, $\lev_6w$.}
\label{alg-lev}
\end{figure}

We will now proceed to proving Theorem~\ref{thm-lev}.

\begin{proof}
Let $j\in\enum1k$.
For $n\in\N$, we abbreviate
\[
    a_j := ( -m_j, -b_j, -l_{1,j-1}, l_{2,j-1}-l_{2,k},
                         -l_{3,j-1}, l_{4,j-1}-l_{4,k},\dotsc),
    \quad
    u_{n-2,j} := \Gol_{a_j}^n(1)
\]
and obtain $\content{u_{n-2,j}} = T_na_j$, thus
\begin{gather}\label{eqn-lu1j}
    (-1)^nl_{n,j} = \content{u_{n-2,1} \cat\dotsm\cat u_{n-2,j}}
                  = \len{u_{n-1,1} \cat\dotsm\cat u_{n-1,j}}
\end{gather}
from Corollary~\ref{cor-Golw}\ref{sub-lenGolzw}
and the assumptions of the theorem.
We want to prove by induction that
\begin{gather}\label{eqn-u1kLnw}
u_{n,1} \cat\dotsm\cat u_{n,k} = \begin{cases}
-\ov{\Lev^nw} & \text{if $n$ is even},\\
-\longrev{\Lev^nw} & \text{if $n$ is odd}
\end{cases}
\end{gather}
holds for all $n\in\N_0$.
We start by observing that, for each $n\in\N_0$,
\eqref{eqn-u1kLnw} implies
\begin{gather}\label{eqn-lenLevnw}
  \len{\Lev^nw=l_{n+1,k}}
\end{gather}
according to~\eqref{eqn-lu1j}, Remarks~\ref{rmk-bar} and~\ref{rmk-rev}.
Furthermore, $u_{0,j}=\Gol_{a_j}^2(1)=(\mi b_j^{\mi m_j})$ entails
$u_{0,1} \cat\dotsm\cat u_{0,k} = -\bar w = -\ov{L^0w}$,
settling assertion~\eqref{eqn-u1kLnw},
and thereby also~\eqref{eqn-lenLevnw}, for $n=0$.

Now let $n\in\N$.
If $n$ is odd, then
$u_{n,j} = \Gol_{-l_{n,j-1}}u_{n-1,j}$, hence
\[
  u_{n,1} \cat\dotsm\cat u_{n,k}
  = \Gol_0(u_{n-1,1} \cat\dotsm\cat u_{n-1,k})
  = \Gol_0\ov{-\Lev^{n-1}w}
  = -\Gol_1\Lev^{n-1}w
  = -\longrev{\Lev^nw}
\]
by~\eqref{eqn-Golzvw}, \eqref{eqn-lu1j},
induction hypothesis~\eqref{eqn-u1kLnw},
Lemma~\ref{lem-Golinvo} and the definition of $\Lev$.

If $n$ is even, then
$u_{n,j} = \Gol_{l_{n,j-1}-l_{n,k}}u_{n-1,j}$, hence
\begin{align*}
  u_{n,1} \cat\dotsm\cat u_{n,k}
  &= \Gol_{-l_{n,k}}(u_{n-1,1} \cat\dotsm\cat u_{n-1,k})
   = \Gol_{-l_{n,k}}\longrev{-\Lev^{n-1}w}
   = \Gol_{-l_{n,k}}\ov{-\Lev^{n-1}w}^{-1}\\
  &= -\Gol_{1+l_{n,k}}(\Lev^{n-1}w)^{-1}
   = -(\Gol_1\Lev^{n-1}w)^{-1}
   = -{\longrev{\Lev^nw}}^{-1}
   = -\ov{\Lev^nw}
\end{align*}
by~\eqref{eqn-Golzvw}, \eqref{eqn-lu1j},
induction hypotheses~\eqref{eqn-u1kLnw}, \eqref{eqn-lenLevnw},
Corollary~\ref{cor-Golw}\ref{sub-Golw-1},
Lemma~\ref{lem-Golinvo} and the definition of $\Lev$.
We have thus proved equation~\eqref{eqn-lenLevnw}
for all $n\in\N_0$,
which is the assertion of the theorem.
\end{proof}

\section{Actual computations}\label{sec-compute}

We will now use our algebraic theory
to compute large terms of some golombic and Levine sequences in the OEIS.
To that end we have implemented Algorithms~\ref{alg-gol} and~\ref{alg-lev}
in PARI~\cite{PARI}.

\begin{table}[!ht]
\centering
\renewcommand{\arraystretch}{0.5}
\begin{tabular}{|r|r|}
\hline
\small$n$  &\multicolumn{1}{c|}{\small$\gol_n(2)$}\\
\hline
\tn{1} & \tn{1}\\
\tn{2} & \tn{2}\\
\tn{3} & \tn{2}\\
\tn{4} & \tn{3}\\
\tn{5} & \tn{5}\\
\tn{6} & \tn{11}\\
\tn{7} & \tn{38}\\
\tn{8} & \tn{272}\\
\tn{9} & \tn{6474}\\
\tn{10} & \tn{1090483}\\
\tn{11} & \tn{4363282578}\\
\tn{12} & \tn{2940715000315189}\\
\tn{13} & \tn{7930047000157075949085439}\\
\tn{14} & \tn{14412592242471457956514645440241289655074}\\
\tn{15} & \tn{70636608026754077888330819116433040562%
              582634705380432362008848092}\\[0.3ex]
\tn{16} & \tn{62919380747847674923268444538883081290%
              15560863748318746359641092668720877652%
              23798623957170361696051714808}\\
\shortstack[r]{\ \\[0.3ex]\tn{17}\\[0ex]}
& \shortstack[l]{%
          \tn{274679710774467801558637071964120644%
              623371548235513562173560586253899638%
              892063967223817009010426152743428280}\\[-.2ex]
      \tn{95015228805612576584237742560138241905230925658418113598362450}}\\[0.3ex]
\shortstack[r]{\tn{18}\\[0.4ex]}
& \shortstack[l]{%
          \tn{1068128249654676059545859772511092830%
              2862782238991152779159513586521119002%
              6888086973173355550098402002269620478}\\[-.2ex]
          \tn{9118657910813125744810794765768812113%
              6855941878034817237323089025802952868%
              9264337149196604264834413862430830118}\\[-.2ex]
          \tn{92049129175381484051105807121010552592345789589177244}}\\[0.4ex]
\shortstack[r]{\tn{19}\\[1ex]}
& \shortstack[l]{%
          \tn{18132694087736252753825746762353797346%
              19275691193110572504441448757346591432%
              78273051378468643572659960483509344293}\\[-.2ex]
          \tn{92650305504635411522647472146929043966%
              38434076992830632424546405639737791621%
              19665115691428170790457955121041161219}\\[-.2ex]
          \tn{23792554965830233832467902012860708233%
              74889716213602135071359209036524798037%
              60317412179872022642510942471562312220}\\[-.2ex]
          \tn{92230645474866085358642626579262442931%
              79785271338594333828913921905524801116%
              72830592742834237563852836}}\\
\hline
\end{tabular}
\caption{first 19 terms of $\gol(2)=$ \OEIS{014644}}
\label{tab-golombic2}
\end{table}

We start with the golombic sequence $\gol(2)=\OEIS{014644}$
mentioned in our introduction and displayed as~\eqref{eqn-gol2}
in Section~\ref{sec-intseq}.
Its first fifteen terms are already known.
Feeding the word $w:=\Gol^{10}(2) = (1^1 2^2 \!\dotsc 1090483^{6474})$
to Algorithm~\ref{alg-gol} produces the golombic numbers
$\gol_1w=\gol_{11}(2),\dotsc,\gol_9w=\gol_{19}(2)$
displayed in Table~\ref{tab-golombic2}.
Other golombic sequences, for example $\gol(3)$
(or even $\gol(-3)$, which has positive and negative terms)
could be easily computed and added to the OEIS if desired.

\begin{table}[!ht]
\centering
\renewcommand{\arraystretch}{0.5}
\begin{tabular}{|r|r|}
\hline
\small$n$  &\multicolumn{1}{c|}{\small$\lev_n(2)$}\\
\hline
\tn{1} & \tn{1}\\
\tn{2} & \tn{2}\\
\tn{3} & \tn{2}\\
\tn{4} & \tn{3}\\
\tn{5} & \tn{4}\\
\tn{6} & \tn{7}\\
\tn{7} & \tn{14}\\
\tn{8} & \tn{42}\\
\tn{9} & \tn{213}\\
\tn{10} & \tn{2837}\\
\tn{11} & \tn{175450}\\
\tn{12} & \tn{139759600}\\
\tn{13} & \tn{6837625106787}\\
\tn{14} & \tn{266437144916648607844}\\
\tn{15} & \tn{508009471379488821444261986503540}\\
\tn{16} &
  \tn{37745517525533091954736701257541238885239740313139682}\\
\tn{17} &
  \tn{534742638381269723378613957622045014225037327749913%
      0252554080838158299886992660750432}\\
\shortstack[r]{\ \\[0.3ex]\tn{18}\\[0ex]}
& \shortstack[l]{%
  \tn{562886959965089436255276483697229698570162841824809622582532%
      6053811392335926611739231671481093004498028362253089934588390}\\[-.2ex]
  \tn{3443973143886461}
  }\\[0ex]
\shortstack[r]{\ \\[0.3ex]\tn{19}\\[0ex]}
& \shortstack[l]{%
  \tn{8394177266373517316056054367253472668387345374746259369127854%
      4525723285290023673872585715830432071384827472565652426695269}\\[-.2ex]
  \tn{7247104588082417791326567485011836725440062543774312172177629%
      64060736471826937656819379445242826439}
   }\\[.4ex]
\shortstack[r]{\tn{20}\\[0.4ex]}
& \shortstack[l]{%
  \tn{13176858547072653394798357452304773863137353203466562795709626%
      36621730054681161341912647579190214493114252604225721726187465}\\[-.2ex]      
  \tn{27090032696189477749573760519660979835010392974960148241997061%
      14081418515499154853768426703009531845240324456626705643814014}\\[-.2ex]
  \tn{49267858165631589878586040172951626441216996746793774353710261%
      882069842922084089160802454747060478632732814946}
}\\
\hline
\end{tabular}
\caption{first twenty terms of $\lev(2) = \OEIS{011784}$}
\label{tab-levine2}
\end{table}

Let us now focus on the Levine sequences in the OEIS,
most importantly Levine's original sequence
$\lev(2)=\OEIS{011784}$.
We mentioned it in our introduction,
as it has been our main motivation for this whole research project.
Let us illustrate how we compute all our Levine numbers from scratch:
We begin with calculating $\lev_1(2),\dotsc,\lev_8(2)$
from~\eqref{eqn-lev2triangle} and~\eqref{eqn-levna}
by hand.
Next we apply Algorithm~\ref{alg-lev} to the words
$w_2:=L^2(2)=(2^1 1^1)$, $w_5:=L^5(2)=(4^1 3^1 2^2 1^3)$, 
$w_8:=L^8(2)=(42^1\!\dotsc 2^{13} 1^{14})$ and 
$w_{11}:=L^{11}(2)=(175450^1\!\dotsc1^{2837})$
one after the other,
each time feeding the Levine numbers returned
by the previous call to the next call.
Thereby, we first produce
$\lev_1w_2=\lev_3(2),\dotsc,\lev_9w_2=\lev_{11}(2)=\lev_6w_5$,
then
$\lev_1w_5=\lev_6(2),\dotsc,\lev_9w_5=\lev_{14}(2)=\lev_6w_8$,
then
$\lev_1w_8=\lev_9(2),\dotsc,\lev_9w_8=\lev_{17}(2)=\lev_6w_{11}$,
and then finally
$\lev_1w_{11}=\lev_{12}(2),\dotsc,\lev_9w_{11}=\lev_{20}(2)$.
In this way, we obtain the first twenty 
terms of Levine's original sequence  $\lev(2)=\OEIS{011784}$
within two hours.

\begin{table}[!ht]
\centering
\renewcommand{\arraystretch}{0.5}
\begin{tabular}{|r|r|}
\hline
\small$n$  &\multicolumn{1}{c|}{\small$\lev_n(0,0,1)$}\\
\hline
\tn{1} & \tn{3}\\
\tn{2} & \tn{1}\\
\tn{3} & \tn{3}\\
\tn{4} & \tn{3}\\
\tn{5} & \tn{6}\\
\tn{6} & \tn{10}\\
\tn{7} & \tn{28}\\
\tn{8} & \tn{108}\\
\tn{9} & \tn{1011}\\
\tn{10} & \tn{32511}\\
\tn{11} & \tn{9314238}\\
\tn{12} & \tn{84560776390}\\
\tn{13} & \tn{219625370880235960}\\
\tn{14} & \tn{5178941522681382123892005221}\\
\tn{15} & \tn{317195599240175645015464306479382985752031865}\\
\tn{16} &
  \tn{458118706320594776183599743881383842326646671002717%
      727944161269026105841}\\
\tn{17} &
  \tn{40524423106475362131212671577710438781983672527594689540186%
      481007941224989794967558352011528561939344387386361918024}\\
\shortstack[r]{\ \\[0.3ex]\tn{18}\\[0ex]}
& \shortstack[l]{%
  \tn{517734706951251237535645078871725288172010082560578109830724%
      732757803546151844793635698385556746175450470169068259400348}\\[-.2ex]
  \tn{809846908421331320764244898886841226930736565024659040954172%
      9350948}
  }\\[0.4ex]
\shortstack[r]{\ \\[0.3ex]\tn{19}\\[0.45ex]}
& \shortstack[l]{%
  \tn{5851089907854437913711197398216048423981863081907791146594428%
      0902674553085487990182624491947708226693686304248007269656384}\\[-.2ex]
  \tn{4454389594821366538242678424804173898851358613606640376458206%
      1629553025446916407883007783604880422644051997090474874507220}\\[-.2ex]
  \tn{4122425938292829190234648392016115743586341696466638099854}
   }\\
\hline
\end{tabular}
\caption{first 19 terms of Levine sequence $\lev(0,0,1)$
         (cf.~\OEIS{061892})}
\label{tab-levine001}
\end{table}

Along the same lines we compute
the first nineteen terms of the Levine sequence $\lev(0,0,1)$,
which is found in the OEIS with an extra zeroth term as~\OEIS{061892},
and the first eighteen terms of the Levine sequence $\lev(0,2)$,
which corresponds to \OEIS{061894}.
All those Levine numbers
are displayed in Table~\ref{tab-levine001}
and in Table~\ref{tab-levine02} respectively.

\begin{table}[!ht]
\centering
\renewcommand{\arraystretch}{0.5}
\begin{tabular}{|r|r|}
\hline
\small$n$  &\multicolumn{1}{c|}{\small$\lev_n(0,2)$}\\
\hline
\tn{1} & \tn{2}\\
\tn{2} & \tn{2}\\
\tn{3} & \tn{4}\\
\tn{4} & \tn{6}\\
\tn{5} & \tn{13}\\
\tn{6} & \tn{35}\\
\tn{7} & \tn{171}\\
\tn{8} & \tn{1934}\\
\tn{9} & \tn{97151}\\
\tn{10} & \tn{52942129}\\
\tn{11} & \tn{1435382350480}\\
\tn{12} & \tn{21191828466255176653}\\
\tn{13} & \tn{8482726531439110654657256441218}\\
\tn{14} & \tn{50131800300416773319763186119561362369281827059942}\\
\tn{15} &
  \tn{118593237444245044162979641011632199299775500898191%
      134931490487767364941673215121}\\
\shortstack[r]{\ \\[0.3ex]\tn{16}\\[0ex]}
& \shortstack[l]{%
  \tn{16580038213765946158970490735716854098151679318563200923947%
      35105253090963710161624288562839274938016433284875404831152}\\[-.2ex]
  \tn{497013018105}
  }\\[0.4ex]
\shortstack[r]{\ \\[0.3ex]\tn{17}\\[0.45ex]}
& \shortstack[l]{%
  \tn{548350247347858284424735841286598432135554883736080031677080%
      628586300038196498320973659796091956479764217227046839806477}\\[-.2ex]
  \tn{827739983749817030539482371451432934487316574132306048207393%
      59759059631007625156541555199}\\[-.2ex]
   }\\
\shortstack[r]{\ \\[0.3ex]\tn{18}\\[0.45ex]}
& \shortstack[l]{%
  \tn{2535457280728691927642079988404251435798134843896320193096475%
      4688452974597655028851506417116883244713034524631462612033806}\\[-.2ex]
  \tn{6989405655081689947162822110618637554083688246295576367009862%
      2518268788598523771527763241733406357760755761701436833329021}\\[-.2ex]
  \tn{6006596822290193802143301167039527782999078163517707768366511%
      964728681044585582158791222111730}
   }\\
\hline
\end{tabular}
\caption{first 18 terms of Levine sequence $\lev(0,2)$
         (cf.~\OEIS{061894})}
\label{tab-levine02}
\end{table}

The reader might notice that the OEIS Names
for the latter two Levine sequences
seem to use different base tuples.
This is because
the OEIS entries for all Levine sequences
reverse each tuple in the Levine triangle
{\em before} applying the golombic operator $\Gol$,
while our Levine operator $L$ first applies $\Gol$
and {\em then} reverses the resulting tuple.
Remember that we chose this standpoint
because it allows us to express the Levine operator on $\N^{(\N_0)}$
in terms of integer partitions and their conjugates.
Other Levine sequences, for example $\lev(1,2)$
(or even $\lev(-4)$, which has positive and negative terms)
could be easily computed and added to the OEIS if desired.

\section{Loose ends}\label{sec-loose}

It turns out that $\gol a$ and $\lev a$ are unbounded
for every tuple $a\in\N_0^{(\N)}$ other than~$()$ or~$(1)$.
On the other hand, many words are taken to the empty word
by the golombic or Levine operator after a few iterations.
For example, the sequences
\begin{align*}
    \gol(5^{\mi1} \mi4^1) &= (0,-9),\\
    \gol(\mi2^1) &= (1,-2,-2,1,-1,-1),\\
    \gol(\mi2^1 3^1 \mi1^1) &= (3,0,1,-2,-2,2,-1,-2),\\
    \lev(\mi2^1) &= (1,-2,-1,1),\\
    \lev(\mi3^1) &= (1,-3,-3,3,-1,-1),\\
    \lev(1^{\mi3}2^2) &= (-1,1,-3,-6,4,-3,-4,3,-1,-1)
\end{align*}
are all finite.
It would be nice to characterise the words that exhibit such a behaviour.
Even more interesting is the question whether there are words whose
Levine or golombic sequence is bounded {\em and} infinite.
We suspect that such a sequence, if existent at all,
would have to become periodic.

Let us call a word $w$ given in reduced form
$w=(b_1^{m_1}\dots b_k^{m_k})$
\df{homogeneous}
if $\{\thru b1k\}$ or $\{\thru{-b}1k\}$
is a subset of $\N_0$
and if either $\{\thru m1k\}$ or $\{\thru{-m}1k\}$
is a subset of $\N$.
If $w$ is homogeneous,
then $\Gol w$ and $\Lev w$
are obviously homogeneous as well.
Concerning the asymptotics of an unbounded
golombic or Levine sequence
based at a homogenous word,
two constants seem to play a role:
the golden ratio $\ph=\frac12(1+\sqrt5)=$ \OEIS{001622}
and Mallows's constant $\kappa=$ \OEIS{369988} $\approx0.278877061$
established in~\cite{MiSa,Mi2,Mi1}.
Part~\ref{sub-limlev} of the following conjecture
generalises a hypothesis expressed by Mallows in~\cite{MaPoSl}.

\begin{cnj}\label{cnj-lim}
Let $w$ be a homogeneous word.
Then
\begin{substate}
\item\label{sub-limgol}
    the golombic sequence $\gol w$ is either bounded or
    $\lim\limits_{n\to\infty}
     \frac{\gol_{n+1}w}{\gol_nw\,\cdot\,\gol_{n-1}w}
     = \ph-1$,
\item\label{sub-limlev}
    the Levine sequence $\lev w$ is either bounded or
    $\lim\limits_{n\to\infty}
     \frac{\lev_{n+1}w}{\lev_nw\,\cdot\,\lev_{n-1}w}
     = \kappa$.
\end{substate}
\end{cnj}

Let us now take a more algebraic look
at the golombic and Levine numbers.
Meditating on Theorems~\ref{thm-gol} and~\ref{thm-lev}
for a moment reveals that
$\gol_n(b_1^{m_1}\dots b_k^{m_k})$
and $\lev_n(b_1^{m_1}\dots b_k^{m_k})$
are integer-valued polynomials in $\thru m1k$, $\thru b1k$
for every fixed $k\in\N$.
Their degrees in these \qt{variables}
seem to be the Fibonacci numbers
\[
    (F_0,F_1,F_2,\dotsc)=\OEIS{000045}=(0,1,1,2,3,5,8,13,21,\dotsc).
\]

\begin{cnj}
For $j,k,n\in\N$ with $j\leq k$,
both polynomial expressions $\gol_n(b_1^{m_1}\dots b_k^{m_k})$
and $\lev_n(b_1^{m_1}\dots b_k^{m_k})$
have degree $F_n$ in $m_j$ and degree $F_{n-1}$ in $b_j$.
\end{cnj}

\begin{OEISentries}

\oeitem{000045} Fibonacci numbers: F(n) = F(n-1) + F(n-2)
    with F(0) = 0 and F(1) = 1.
    
\oeitem{001462} Golomb's sequence:
    a(n) is the number of times n occurs, starting with a(1) = 1.

\oeitem{001463} Partial sums of \OEIS{001462};
    also a(n) is the last occurrence of n in \OEIS{001462}.

\oeitem{001622} 	Decimal expansion of golden ratio
  phi (or tau) = (1 + sqrt(5))/2.

\oeitem{008275} 	Triangle read by rows of Stirling numbers of first kind,
  s(n,k), n \>= 1, 1 \<= k \<= n.

\oeitem{008277} Triangle of Stirling numbers of the second kind,
  S2(n,k), n \>= 1, 1 \<= k \<= n.

\oeitem{011784} Levine's sequence.
  First construct a triangle as follows. Row 1 is \{1,1\};
  if row n is \{r\_1, ..., r\_k\} then row n+1 consists of
  \{r\_k 1's, r\_\{k-1\} 2's, r\_\{k-2\} 3's, etc.\};
  sequence consists of the final elements in each row.

\oeitem{012257} Irregular triangle read by rows:
  row 0 is \{2\}; if row n is \{r\_1, ..., r\_k\}
  then row n+1 is
  \{r\_k 1's, r\_\{k-1\} 2's, r\_\{k-2\} 3's, etc.\}.

\oeitem{014643} Triangular array starting with \{1,1\};
    then i-th term in a row gives number of i's in next row.

\oeitem{014644} Form array starting with \{1,1\};
    then i-th term in a row gives number of i's in next row;
    sequence is formed from final term in each row.

\oeitem{061892} Lionel-Levine-sequence generated by (1,0,0).

\oeitem{061894} Lionel-Levine-sequence generated by (2,0).

\oeitem{369988} Decimal expansion of Mallows's constant
  or stribolic constant kappa (of order~1).

\end{OEISentries}

\end{document}